\documentclass[12pt]{article}
\usepackage{graphicx}
\usepackage{amssymb,amsmath,amsthm}
\usepackage{enumerate}
\usepackage{color}

\usepackage{epsfig} 

\usepackage{latexsym}

\newtheorem{theorem}{Theorem}

\newtheorem{lemma}[theorem]{Lemma}
\newtheorem{proposition}[theorem]{Proposition}
\newtheorem{remark}[theorem]{Remark}

\title{Carr-Nadtochiy's weak reflection principle for Markov chains on $ \mathbf{Z}^d $}
\author
{
Yuri Imamura 
\footnote{Dept. of Business Economics, Tokyo University of Science, Japan, (mail: imamuray@rs.tus.ac.jp)}
}

\begin{document}

\maketitle

\begin{abstract} 
The present paper establishes a 
discrete version of the result
obtained by P. Carr and S. Nadtochiy in \cite {CN} for $1$-dimensional diffusion processes.  
Our result is for Markov chains 
on $ \mathbf{Z}^d $. 
\end{abstract}

\section{Introduction}
Let $ X^x $ be a diffusion  
process on $ \mathbb{R} $
starting from $ x $, 
and $\tau_y$ be the first time when $X$ visits $ y $, 
which is smaller than $ x $, that is,  
\begin{equation*}
\tau_y := \inf \{s>0: X^x_s \leq y \}.
\end{equation*}
In \cite{CN}, 
it is proven that,
$ X $ being in 
quite a general class,  
for a measurable function $ f $ which is zero on $ (-\infty, y] $ with some regularity conditions, there exists a function $ g $
which is zero on $ [y, \infty) $ such that for any $t>0$, 
\begin{equation}\label{eq-BNtrans}
\mathbf{E} [ f(X^x_t) 1_{\{\tau_y > t\}}]
=  E[g(X^x_t)]. 
\end{equation}
We call the correspondence
\begin{equation*}
f \mapsto g 
\end{equation*}
Carr-Nadtochiy transform.

The financial meaning of 
Carr-Nadtochiy transform
is 
as follows. 
\if1 
In this note, we consider a discrete version of 
Bayraktar-Nadtochiy transform:
let $ Y^\beta  $ be a 
Markov chain on $ \mathbf{Z} $ defined by
\begin{equation*}
P (Y^\beta_{k+1} = Y^\beta_{k} \pm 1| Y^\beta_{k}, \cdots, Y^\beta_0 )
= \frac{ 1 \pm \beta (Y_k) }{2},
k=0,1,\cdots, 
\end{equation*}
where $ \beta $ is a function on 
$ \mathbf{Z} $ taking values
in $ (-1,1) $,
and set
$ \tau^{\beta} $ to be 
the first hitting time of $ Y^\beta $ 
to $ 0 $:
\begin{equation*}
\tau^\beta = 
\min \{ k \in \mathbb{Z}_+ :
Y^\beta_k = 0 \}. 
\end{equation*}
We are interested in 
the existence and uniqueness of
$ T f $ for a function $ f $ on
$ \mathbf{N} $ with the following property
\begin{equation*}
\mathbf{E} [ f(Y^\beta) 1_{\{\tau^\beta > t\}}]
= E[Tf(Y^\beta_k)]
\end{equation*}
and $ Tf (x) = 0 $ for $ x \leq 0 $. 

When $ \beta \equiv 0 $, 
i.e. $ Y^\beta = Y^0 $ is a simple random walk, 
the transform is given by 
the reflection principle:
\begin{equation*}
T f (x) = f(-x).
\end{equation*}
\fi 
We consider 
a down and out option whose payoff at the maturity $T$ is given by  
\begin{equation*}
f(X^x_T) I_{\{\tau_y >T\}}, 
\end{equation*}
where $X_T$ is the stock price at the maturity, $y\ (\in (0,X_0))$ is a knock-out (lower) boundary, 
and $f$ is a payoff function of this option. 
We choose a portfolio of European type options with payoff $f$ 
and $-g$, where $ f \mapsto g $ is given by Carr-Nadtochiy transform.
Then it holds that,
for any $0<t<T$, 
\begin{equation}\label{1021-1}
\begin{split}
&\mathbf{E} [e^{-r(T-t)}  f(X^x_T) I_{\{\tau_y >T\}} \ | \mathcal{F}_t]
\\&=I_{\{t \leq T \wedge \tau_y \} }\mathbf{E} [e^{-r(T-t)}  f(X^x_T) I_{\{\tau_y >T\}} \ | \mathcal{F}_t]
\\& =I_{\{t \leq T \wedge \tau_y \} }\mathbf{E} [e^{-r(T-t)}  f(X^x_T) \ |\mathcal{F}_t]
-I_{\{t \leq T \wedge \tau_y \} }\mathbf{E} [e^{-r(T-t)}  g(X^x_T) \ |\mathcal{F}_t],
\end{split}
\end{equation}
where $r$ is the risk-free interest rate. 
The equation \eqref{1021-1}
shows that the down and out option can be hedged by
the \underline{static} portfolio of the European plain option with pay-off $ f $ and 
that with $ -g $ since
\begin{itemize}
\item if $ X $ never hit $ y $
until the maturity, 
the portfolio at the maturity 
pays $ f $, which hedges the down and out option that is active.
\item Once $ X $ hits $ y $, 
the hedger should liquidate the portfolio at $ \tau_y $. Thanks to \eqref{eq-BNtrans} with strong Markov property of $ X $, it costs zero. since the pay-off at the maturity is also zero, it is hedged.  
\end{itemize}

In \cite{CN} 
an analytic 
form of $g$ in \eqref{eq-BNtrans} for 
a class of $1$-dimensional diffusion processes 
whose volatility coefficients $\sigma$ and 
drift coefficients $\mu$, 
with some regularity conditions.
Without loss of generality the knock out (or knock in) boundary can be $\{0\}$. 
\begin{theorem}[Proposition $1$ in \cite{BN} and Theorem $2.7$ in \cite{CN}]\label{CN-CB}
Let $X$ be a diffusion process with a regularity condition, and $f$ be a function 
with $ \mathrm{supp} \, f  = [0, \infty) $ whose derivative is locally integrable.
Then there exists 
a continuous and exponentially bounded function $g$ with $
\mathrm{supp}\, g= (-\infty, 0] $ 
satisfying \eqref{eq-BNtrans}, 
which is given by 
\begin{equation*}
\begin{split}
g(x)
&= 
\frac{2}{\pi i } \int_{\gamma -i \infty}^{\gamma +i \infty}
\frac{w \psi_1(x,w)}{\partial_x \psi_1(o,w) -\partial_x \psi_2(o,w)} \\
& 
\int_{-\infty}^{0} \frac{\psi_1(z,w)}{\sigma^2(z)}
\exp 
\left(-2 \int_0^z \frac{\mu (y)}{\sigma^2(y)}\ dy \right)
f(z) \, dz \, dw
\end{split}
\end{equation*}
for any large enough $\gamma >0$, where $\psi_1$ and $\psi_2$ are the fundamental solution of the following 
Strum-Liouville equation 
\begin{equation*}
\frac{1}{2} \sigma^2(x) 
\frac{\partial^2}{\partial x^2} \psi (x,w) +\mu (x) \frac{\partial}{\partial x} 
\psi (x,w) -w^2\psi(x,w) =0
\end{equation*}
such that $\psi_1$ is square integrable function on $(-\infty, 0)$,
$\psi_2$ is square integrable function on $(0,\infty )$ 
and 
$\psi_1(0,w ) = \psi_2(0,w) =1$. 
\end{theorem}

The Carr-Nadtochiy transform is 
looked upon as a generalization 
of the reflection principle. 
Let us explain the reason why. 
We say that a $1$-dimensional strong Markov process satisfies
Put-Call symmetry\footnote{Put-Call symmetry of a diffusion process was discussed in \cite{CL}. 
A sufficient condition of Put-Call symmetry is that the diffusion and the drift coefficients are symmetric 
with respect to the boundary. 
Most of strong Markov processes 
including the ones used in financial modeling,   
however, do not satisfy the Put-Call symmetry for any points.
In \cite {IIKO},
introduced is a scheme for {\em symmetrization}, with which
a diffusion is transformed to 
one with Put-Call symmetry. 
The scheme gives a 
numerical framework to
calculate the price of a barrier option.}
at $y$ if the law of $X_t$ and that of $2y-X_t$ coincide 
for any $t>0$ when $X$ starts from $y$. 
If $ X $ is with the put-call symmetry, 
Carr-Nadtochiy transform is given by $f(2y-x)$. 
The most important example 
with the Put-Call symmetry 
is $1$-dimensional Brownian motion. 
Actually, 
Brownian motion satisfies Put-Call symmetry for any point
thanks to the reflection principle. 

As we have seen, the put-call symmetry or the Carr-Nadtochiy transform allow us to construct a static hedge of barrier option. 
A multi-dimensional extension becomes very difficult, mainly because 
the boundary is not anymore a single point. 
Construction of Carr-Nadtochiy transform
for a multi-dimensional process 
is, to say nothing of, among the hardest.  
The main result (Theorem \ref{theorem7}) of the present paper therefore could be a step forward since the case where the boundary is an infinite set is managed.
The proof only needs basic toolkits from linear algebra. 
The result is actually a discovery rather than an invention. 

From a practical point of view, 
the explicit expression (Proposition \ref{prop}) will be useful to construct 
a specific static hedge in a multi-dimensional setting including stochastic volatility environment, 
which, though, we do not discuss
in detail in the present paper.

The organization of the present paper is quite simple. In the following
section, after a preliminary subsection where a key lemma is proven, a proof
to the main theorem, based on an explicit construction of the transformation,
is given.

\section{Multi-dimensional Discrete Carr-Nadtochiy Transform}

\subsection{Preliminaries}
Let $d \in \mathbf{N}$. 
We set $B=\{ x = (x_1, \cdots , x_d) \in \mathbf{Z}^d;\ x_d >0\}$ and  
$\partial B= \{ x\ \in \mathbf{Z}^d;\ x_d=0\}$. 
We will obtain 
a discrete-time analogue of 
Carr-Nadtochiy transform 
when the process is a random walk, meaning that it is Markovian and 
at each time the process 
can move in the direction of 
the unit vector
$e_i= (0,\cdots , 0 , 1, 0, \cdots, 0)$, $ i=1, \cdots, d$, 
or to be precise, it is a
discrete-time time-homogeneous
Markov chain on $ \mathbf{Z}^d $ 
with starting at $x \in \mathbf{Z}^d$ 
which we will denote by
$ (Z_t^x)_
{t \in \mathbf{Z}_+ } $,
with 
\begin{equation*}
\sum_{i=1}^dP(Z_1^x=x  +e_i )  +\sum_{i=1}^dP(Z_1^x=x   - e_i) =1, 
\end{equation*} 
for any $x \in \mathbf{Z}^d$.
Here we further assume that
\begin{equation}\label{nondeg}
P(Z_1^x=x  +e_i),\ P(Z_1^x=x   - e_i) >0\ (i=1,\cdots, d) 
\end{equation} 
for any $x \in \mathbf{Z}^d$.

Let $\tau_{\partial B}$ be the first time when $X$ visit $\partial B$, that is, 
\begin{equation}
\tau_{\partial B} := \inf \{s > 0:\ Z_s^x = \partial B\}. 
\end{equation}

$Z_t^x \in \mathcal{S}(t,x) $, 
where 
\begin{equation*}
\begin{split}
\mathcal{S}(t,x) 
 = \{ (s, y) \in \mathbf{N} \times \partial B 
:\ &s \in \{1,\cdots, t\} ,\
|| y || =\sum_{i=1}^{d-1} |k_i| \leq t-s, 
\\& \qquad  
\text{$t-s-||y-x||$ is an even number}\} , 
\end{split}
\end{equation*} 
for each $t\in \mathbf{N}$ and $x \in \partial B$. 

\begin{lemma}\label{cover2}
For any $ (t,x),\  (s, y)\in \mathbf{N} \times \partial B$, 
\begin{equation*}
P(Z_t^x= y \pm s e_d  ) >0 
\end{equation*}
if and only if $ (s,y) \in \mathcal{S}(t,x) $. 
\end{lemma}
\begin{proof}
It is immediate from \eqref{nondeg}.
\end{proof}
For $A \subset \mathbf{Z}^d$, denote by $\mathcal{M}_A^d$ the set of all functions 
$\mathbf{Z}^d \to \mathbf{R}$ with $\mathrm{supp} f \subset A$. 
Let $W^+_y : \mathcal{M}_B^d \to \mathcal{M}_{\mathbf{N}}^1$ and
$W^-_y : \mathcal{M}_{B^c \backslash \partial B}^d  \to \mathcal{M}_{\mathbf{N}}^1$
be defined by
\begin{equation*}
W^{+}_x h (t) = \mathbf{E}[h(Z^x_t) ], 
\end{equation*}
and 
\begin{equation*}
W^{-}_x h (t) = \mathbf{E}[h(Z^x_t)].
\end{equation*}
By Lemma  \ref{cover2}, 
$W^+_x f (t) $ is expressed by 
\begin{equation*}
W^+_x h (t)  = \sum_{(s,y) \in \mathcal{S}(t,x) } h(y + s e_d) P(Z_t^x =y + s e_d),
\end{equation*}
and 
\begin{equation*}
W^-_x h (t)  = \sum_{(s,y) \in \mathcal{S}(t,x) } h(y - s e_d) P(Z_t^x =y - s e_d). 
\end{equation*}

For each $ (t,x) \in \mathbf{N} \times \partial B$, we associate ``square matrices"
$ W^{\pm}_{t,x} $ as follows. 
Let 
\begin{equation*}
\mathcal{L}_{(t,x)}^{+}
:= \{ \{ h (z \pm ue_d) \}_{(u,z) \in \mathcal{S}(t,x) } \in \mathbf{R}^{\sharp \mathcal{S}(t,x)} : h \in \mathcal{M}_B^d  \}
\end{equation*}
and 
\begin{equation*}
\mathcal{L}_{(t,x)}^{-}
:= \{ \{ h (z+ue_d) \}_{(u,z) \in \mathcal{S}(t,x) } \in \mathbf{R}^{\sharp \mathcal{S}(t,x)} : h \in \mathcal{M}_{B^c \backslash \partial B}^d \}.
\end{equation*}
Define $ W^{\pm}_{t,x} $ by
\begin{equation*}
W^{\pm}_{t,x}  h (s,y) = W^{\pm}_y h (s)  = \sum_{(u,z) \in \mathcal{S}(s,y) } h(z \pm u e_d) P(Z_s^y =z \pm u e_d) 
\end{equation*}
for $(s,y) \in \mathcal{S}(t,x)$. 
Since $P(Z_s^y =z + u e_d) =0$ for $(u,z) \notin \mathcal{S}(s,y) $, 
we can regard  $W^{\pm}_{t,x} $
as a linear map 
from $ \mathcal{L}_{(t,x)}^{\pm} \simeq 
\mathbf{R}^{\sharp \mathcal{S}(t,x)} $
to $ \mathcal{M}^{d+1}_{\mathcal{S}(t,x)
} \simeq \mathbf{R}^{\sharp \mathcal{S}(t,x)}$, and therefore, we can identify it
with a square matrix, 
which is in fact invertible.  
\begin{lemma}\label{Lem-det}
For any $ (t,x) \in \mathbf{N} \times \partial B$, 
the determinant of $W^{\pm}_{t,x}$ is given by 
\begin{equation}\label{det-W}
\det W^{\pm}_{t,x}  = \prod_{(s,y) \in \mathcal{S}(t,x)}P(Z_s^y =y \pm s e_d)>0.
\end{equation}
\end{lemma}
To prove Lemma \ref{Lem-det}, we need the following
\begin{lemma}\label{Lem-2-0}
For a finite subset  $ \mathcal{A} \subset \mathbf{N} \times \partial B$, we denote by 
$\mathrm{Sym}(\mathcal{A})$  
the symmetric group over $\mathcal{A}$, and for $ \sigma  \in \mathrm{Sym}(\mathcal{A}) $, 
we write
\begin{equation*}
\sigma(s,y) = (\sigma_1 (s,y),\ \sigma_2(s,y)), \quad (s,y) \in \mathcal{A}. 
\end{equation*}
For any $ \sigma \in  \mathrm{Sym}(\mathcal{A}) $, the following conditions are equivalent; 
\begin{enumerate}
\item[\rm{(i)}]
\begin{equation*}\label{equiv-2+}
\prod_{(s,y) \in \mathcal{A}}P(Z_s^y= \sigma_2 (s,y) + \sigma_1(s,y)  e_d  ) > 0.
\end{equation*}
\item[\rm{(ii)}] 
\begin{equation*}\label{equiv-2+}
\prod_{(s,y) \in \mathcal{A}}P(Z_s^y= \sigma_2 (s,y) - \sigma_1(s,y)  e_d  ) > 0.
\end{equation*}
\item[\rm{(iii)}]
$\sigma (s,y) \in \mathcal{S}(s,y)
$ for all $(s,y) \in \mathcal{A}$.
\item[\rm{(iv)}] 
$\sigma$ is the identity, i.e., $\sigma(s,y)=(s,y)$ for $(s,y) \in \mathcal{A}$.
\end{enumerate}
\end{lemma}
\begin{proof}
By Lemma \ref{cover2}, 
(i), (ii) and (iii) are equivalent. 
Since $(s,y) \in \mathcal{S}(s,y) $ for any $(s,y) \in\mathbf{N} \times \partial B$, (iv) leads to (iii). 
It then remains to prove that (iii) implies (iv). 
Suppose that (iii) is satisfied
but $\sigma$ is not the identity. 
Then the set 
$A:= \{(s,y) \in \mathcal{A}:\ \sigma(s,y) \neq (s,y)\}$ is not empty.  Let $m := \max_{(s,y) \in A} s$. 
We will show that for any $(m,w) \in A$, there is no $(s,y) \in \mathcal{A}$ such that $\sigma (s,y) = (m,w)$, which is a contradiction. 
Since  $\sigma (m,w) \in \mathcal{S}(m,w) \backslash  (m,w) $ for any $w$ with  $(m,w) \in A$, 
we have that $\sigma_1 (m,w)  \neq m$. 
Moreover, for $(\tilde{m},w) \in A$ with $\tilde{m} \neq m$, 
since $\sigma (\tilde{m},w) \in \mathcal{S}(\tilde{m},w)$ is assumed, 
we see that $\sigma_1 (\tilde{m},w) \leq \tilde{m} $. 
Therefore $ \sigma_1 (s,y) \neq m $ for any $(s,y) \in A$.
\end{proof}
\begin{proof}[Proof of Lemma \ref{Lem-det}] 
By the definition of the determinant of a matrix, 
we have that
\begin{equation*}
\begin{split}
\det W^{\pm}_{t,x} & =
  \sum_{\sigma  \in \mathrm{Sym}(\mathcal{S}(t,x))} 
 \mathrm{sgn} (\sigma) 
\prod_{(s,y) \in \mathcal{S}(t,x)} 
P(Z_{s}^y=\sigma_2(s,y) \pm \sigma_1(s,y)  e_d ).
\end{split}\end{equation*}
Now \eqref{det-W} is clear by Lemma \ref{Lem-2-0}. 
\end{proof}

\subsection{Construction of the transform}
Since 
the matrix $W^{\pm}_{t,x}$
 is invertible by Lemma \ref{Lem-det}, 
 we can define a square matrix
\begin{equation*}
\mathcal{N}_{t,x} := (W^{-}_{t,x})^{-1} W^{+}_{t,x}, 
\end{equation*}
which is seen as a linear map 
from 
$ \mathcal{L}^+_{(t,x)}$ to $ \mathcal{L}^-_{(t,x)}$. 
The following lemma is essential for our result.  
\begin{lemma}\label{Lem-5}
For $(t,x)$ and $(\tilde{t},\tilde{x})\in \mathbf{N} \times \partial B$ with 
$ (\tilde{t},\tilde{x})  \in \mathcal{S} (t,x) $, it holds that 
\begin{equation*}
\mathcal{N}_{t,x} 
 = \mathcal{N}_{\tilde{t},\tilde{x}}, \qquad 
 \text{on $\mathcal{L}^+_{(\tilde{t},\tilde{x})}$}.
\end{equation*}
\end{lemma}
\begin{proof}
Let $h^{\pm} \in \mathcal{L}^{\pm}_{(\tilde{t},\tilde{x})}$ and $(s,y) \in \mathcal{S} (\tilde{t},\tilde{x}) $. 
By the definition of $W^{\pm}_{y}$ and $W^{\pm}_{t,x}$, it holds that  
\begin{equation*}
W^{\pm}_{t,x}  h^{\pm} (s,y) = W^{\pm}_y h^{\pm} (s)  = W^{\pm}_{\tilde{t},\tilde{x}}  
h^{\pm} (s,y)
. 
\end{equation*}
Therefore we see   
that $(W^{\pm}_{t,x})^{-1}  = (W^{\pm}_{\tilde{t},\tilde{x}})^{-1} $ on 
$\mathcal{M}_{\mathcal{S} (\tilde{t},\tilde{x})}^{d+1} $, and 
hence 
\begin{equation*}
\begin{split}
\mathcal{N}_{t,x} 
&=  (W^{-}_{t,x})^{-1} W^{+}_{t,x} 
\\&=  (W^{-}_{\tilde{t},\tilde{x}})^{-1} W^{+}_{\tilde{t},\tilde{x}}  
\\& = \mathcal{N}_{\tilde{t},\tilde{x}} .
\end{split}
\end{equation*}
\end{proof}

Since $ \mathcal{L}_{(t,x)}^{+} $ 
(resp. $ \mathcal{L}_{(t,x)}^{-} $)
is projected from 
$ \mathcal{M}_B^d $ (resp. $ \mathcal{M}_{B^c \backslash \partial B}^d $), 
we can extend $ \mathcal{N}_{t,x} $
to a map from $ \mathcal{M}_B^d $ to 
$ \mathcal{M}_{B^c \backslash \partial B}^d $. 
Define $ \mathcal{N} :\mathcal{M}_B^d \to \mathcal{M}_{B^c \backslash \partial B}^d $ 
by
\begin{equation*}
\mathcal{N}h(x-t e_d) := \mathcal{N}_{t,x} h (x-t e_d) 
\qquad \text{for $ (t,x) \in \mathbf{N}  \times \partial B$} .
\end{equation*}
\begin{lemma}\label{Lem-6}
For $ (t,x) $ and $ (\tilde{t},\tilde{x}) \in \mathbf{N} \times \partial B$ 
such that $(\tilde{t},\tilde{x}) \in  \mathcal{S}(t,x)$, 
we have that 
\begin{equation*}
\mathcal{N}h(\tilde{x}-\tilde{t} e_d) = \mathcal{N}_{t,x} h (\tilde{x}-\tilde{t} e_d) .
\end{equation*} 
\end{lemma}
\begin{proof}
It is immediate by Lemma \ref{Lem-5}. 
\end{proof}

The following is our main result. 
\begin{theorem}\label{theorem7}
For $(t,x) \in \mathbf{N} \times \partial B$ and 
$ f\in \mathcal{M}_B $, we have that   
\begin{equation}\label{thm-e}
\mathbf{E} [ f(Z^x_t) ]
=  E[\mathcal{N}f (Z^x_t)].
\end{equation}
\end{theorem}
\begin{proof}
By the definition of $W_x^{\pm}$ and $W_{t,x}^{\pm}$, we know that 
for $h^+ \in \mathcal{M}_B^d$ and $h^- \in \mathcal{M}_{B^c \backslash \partial B}^d$,  
\begin{equation}\label{thm-eq}
\mathbf{E} [ h^{\pm}(Z^x_t) ]
= W_x^{\pm} h^{\pm}(t)= W_{t,x}^{\pm} h^{\pm}(t,x). 
\end{equation}
Since 
\begin{equation*}
W_{t,x}^{+} =W^{-}_{t,x}(W^{-}_{t,x})^{-1}W_{t,x}^{+}
=W^{-}_{t,x}\mathcal{N}_{t,x}, 
\qquad \text{on } \mathcal{L}_{(t,x)}^+
\end{equation*}
we have that 
\begin{equation}\label{eq-BNtrans3}
\begin{split}
\mathbf{E} [ f(Z^x_t) ]
&=W^{+}_{t,x} f (t,x)
\\& =  W^{-}_{t,x}(W^{-}_{t,x})^{-1}W^{+}_{t,x} f (t,x) 
\\& =  W^{-}_{t,x}\mathcal{N}_{t,x} f (t,x). 
\end{split}
\end{equation}
On the other hand, by \eqref{thm-eq} for $ W^{-}_{t,x} $, we then obtain
\begin{equation}\label{eq-thm71}
W^{-}_{t, x}\mathcal{N}_{t,x} f (t,x)
=
W^{-}_{x}\mathcal{N}_{t,x} f(t) 
= \mathbf{E}
[\mathcal{N}_{t,x} f (Z_t^x)]
.
\end{equation}
Thanks to Lemma \ref{Lem-6}, we notice that
\begin{equation}\label{eq-thm7}
\begin{split}
\mathbf{E}[\mathcal{N}_{t,x} f (Z_t^x)]
&= \sum_{(s,y) \in \mathcal{S} (t,x) }\mathcal{N}_{t,x} f (y-se_d) P(Z_t^x = y-se_d)
\\&= \sum_{(s,y) \in \mathcal{S} (t,x) }\mathcal{N}f (y-se_d) P(Z_t^x = y-se_d) .
\\&=\mathbf{E}[\mathcal{N} f (Z_t^x)]
\end{split}
\end{equation}
By combining 
\eqref{eq-BNtrans3}, \eqref{eq-thm71} and \eqref{eq-thm7},
we have the assertion. 
\end{proof}

\subsection{Uniqueness} 
The map $ \mathcal{N} $ could be called 
Carr-Nadtochiy transform for 
the Markov chain $ Z $.
We can prove the uniqueness of the transform:
\begin{theorem}
If a map $\mathcal{N}' :\ \mathcal{M}_{B}^d \rightarrow \mathcal{M}_{B^c \backslash \partial B }^d$
satisfies \eqref{thm-e}
for any $ (t,x) \in \mathbf{N}  \times \partial B $, 
then $ \mathcal{N}'=  \mathcal{N} $. 
\end{theorem}
\begin{proof}
Fix arbitrary $(t,x) \in \mathbf{N} \times \partial B$ and $f \in \mathcal{M}_{B}^d$. Since 
$\mathcal{N}$ and $\mathcal{N}'$ satisfy \eqref{thm-e}, we have that 
\begin{equation*}\label{thm8-eq1}
W_{t,x}^{-} \mathcal{N}'f(t,x) = 
\mathbf{E} [ \mathcal{N}'f(Z^x_t) ]= \mathbf{E} [ f(Z^x_t) ],
\end{equation*}
and 
\begin{equation*}\label{thm8-eq2}
W_{t,x}^{-} \mathcal{N}f(t,x) = 
\mathbf{E} [ \mathcal{N}f(Z^x_t) ]= \mathbf{E} [ f(Z^x_t) ].
\end{equation*}
Therefore we obtain that 
\begin{equation*}\label{thm8-eq2}
W_{t,x}^{-} \mathcal{N}'f(t,x)
=W_{t,x}^{-} \mathcal{N}f(t,x).
\end{equation*}
Since the matrix $W_{t,x}^{-}$ is invertible by Lemma \ref{Lem-det}, we conclude that 
\begin{equation*}
\begin{split}
\mathcal{N}'f(x - t e_d)& = 
(W_{t,x}^{-} )^{-1} W_{t,x}^{-} \mathcal{N}'f(t,x)
\\& =(W_{t,x}^{-} )^{-1} W_{t,x}^{-} \mathcal{N}f(t,x)
\\& =\mathcal{N}f(x - t e_d).
\end{split} 
\end{equation*}
\end{proof}

\subsection{An explicit form}
An explicit form of
$ \mathcal{N} $ 
is given as follows: 
\begin{proposition}\label{prop}
We have that 
\begin{equation*}
\mathcal{N} f(x - t e_d)\\
 = 
\sum_{(s,y) \in \mathcal{S}(t,x)} c_{t,x}(s,y) f(y+se_d),\qquad 
\text{ $(t, x)  \in \mathbf{N} \times \partial  B$ and $f\in\mathbf{M}_{B^c \backslash \partial B}$},
\end{equation*}
where
\begin{equation}\label{def-c-d}
\begin{split}
c_{t,x}({s,y}) &=\frac{1}{\prod_{(l,w) \in \mathcal{S}(t,x)} P(Z_l^{w}= w -l e_d ) } 
\\& \qquad\times   \sum_{\sigma  \in {\rm Sym}(\mathcal{S}(t,x))} 
{\rm sgn} (\sigma) 
P(Z_{\sigma_1(t,x)}^{\sigma_2(t,x)}= y +s e_d )
\\& \qquad \times  
\prod_{(h,z) \in \mathcal{S}(t,x) \backslash \{(t,x)\}} P(Z_{\sigma_1(h,z)}^{ \sigma_2(h,z)}= z -h e_d),
\end{split}
\end{equation}
for $(s,y)  \in \mathbf{N}\times \partial B$.
\end{proposition}
\begin{remark}
Thanks to Lemma \ref{cover2}, \eqref{def-c-d} is well-defined since $\prod_{(l,w) \in \mathcal{S}(t,x)} P(Z_l^{w}= w -l e_d ) $ is not zero. 
\end{remark}
\begin{proof}[Proof of Proposition \ref{prop}]
Recall that 
\begin{equation*}
\mathcal{N}f(x-t e_d) =(W^{-}_{t,x})^{-1} W^{+}_{t,x}f (t,x) 
.\end{equation*}
Here we note that the matrices are given by 
\begin{equation*}
W^{\pm}_{t,x} = \{P(Z_{s}^y = y' \pm s' e_d):\ {(s,y),\ (s',y') \in \mathcal{S}(t,x)}\}. 
\end{equation*}
By Lemma \ref{Lem-det} and Cramer's rule, we have the assertion. 
\end{proof}

\if1 
We note that $c_{t,x}({s,y})=0$ for  $(s,y) \notin  \mathcal{S}(t,x) $ by Lemma \ref{cover2} 
since $ \mathcal{S} (\sigma(t,x)) )$ is a subset of  $\mathcal{S}(t,x)$ for any $\sigma  \in {\rm Sym}(\mathcal{S}(t,x))$.
\fi
\section*{Acknowledgement}
This research was supported by JSPS KAKENHI Grant Number $24840042$.



\begin{thebibliography}{99}

\bibitem{AI}
Akahori, J.  and Imamura, Y.  (2014)
"On a symmetrization of diffusion processes", {\it 
Quantitative Finance} 
14 (7), 1211-1216.

\bibitem{BN}
Bayraktar, E. and Nadtochiy, S. (2015)
"Weak reflection principle for L\'evy processes", 
{\it Annals of Applied Probability}, 25(6), 3251-3294.

\bibitem{CL}
Carr, P. and Lee, R.   (2009)
"Put-call symmetry: Extensions and applications" {\it Math. Finance}, 
19, 523-560. 

\bibitem{CN}
Carr, P. and Nadtochiy, S.  (2011)
"Static Hedging under Time-Homogeneous Diffusions", 
{\it SIAM Journal on Financial Mathematics}, 2, 1, Dec., 2011, 794-838.

\bibitem{IIKO}
Imamura, Y.,
Ishigaki, Y., and Okumura, T. (2014)
"A numerical scheme based on semi-static hedging strategy", 
{\it Monte Carlo Methods and Applications}, 20 (4), 223-235.

\end{thebibliography}
\end{document}